\documentclass[a4paper,11pt]{article}
\usepackage[unicode,hyperfootnotes=false,hidelinks,hyperindex=false]{hyperref}
\newcommand{\supr}{\vee}

\newcommand{\Y}{\bigwedge}
\newcommand{\limp}{\rightarrow}

\usepackage{amsmath}
\usepackage{tikz}
\usetikzlibrary{arrows,shapes}

\DeclareMathOperator{\Hom}{\mathrm{Hom}}

\DeclareMathOperator{\Con}{\mathrm{Con}}
\newcommand{\Cg}{\mathrm{Cg}}

\newcommand{\modelo}[1]{#1}

\newcommand{\variedad}[1]{\mathcal{#1}}
\newcommand{\categoria}[1]{\mathsf{#1}}

\newcommand{\A}{\modelo{A}}
\newcommand{\B}{\modelo{B}}
\newcommand{\bF}{\modelo{F}}

\newcommand{\X}{\modelo{X}}
\renewcommand{\Y}{\modelo{Y}}
\newcommand{\K}{\variedad{K}}

\newcommand{\sbq}{\subseteq}

\usepackage[numbers]{natbib}
\usepackage[utf8]{inputenc}
\usepackage{enumerate}
\usepackage{amsthm}
\usepackage{amssymb} %
\usepackage{amsmath}
\usepackage{verbatim}
\newcommand{\llangle}{\langle\langle}
\newcommand{\rrangle}{\rangle\rangle}
\usepackage{stmaryrd} %

\renewcommand{\modelo}[1]{\mathbf{#1}}
\newcommand{\lb}{\langle}
\newcommand{\rb}{\rangle}

\renewcommand{\phi}{\varphi}

\renewcommand{\vec}[1]{\bar{#1}}

\newcommand{\defi}{\mathrel{\mathop:}=}

\newcommand{\thr}{\mathrel{\theta}}

\DeclareUnicodeCharacter{00B7}{\cdot}
\makeatletter
\DeclareRobustCommand\thinbfseries{%
  \not@math@alphabet\bfseries\mathbf
  \fontseries b\selectfont
}
\makeatother
\renewcommand{\bf}{\thinbfseries}

\newcommand{\axiom}[1]{{\upshape\bf #1}}
\newcommand{\dist}{\axiom{dist}}  
\newcommand{\produ}{\axiom{prod}}
\newcommand{\comm}{\axiom{comm}}
\newcommand{\perm}{\axiom{Perm}}
\newcommand{\ori}{\axiom{Ori}}
\newcommand{\onto}{\axiom{Onto}}
\newcommand{\Abs}{\axiom{Abs}}
\newcommand{\Modi}{\axiom{Mod1}}
\newcommand{\Modii}{\axiom{Mod2}}
\newcommand{\exi}{\axiom{Exi}}  
\newcommand{\puno}{\axiom{p$\mathbf{_1}$}}
\newcommand{\pdos}{\axiom{p$\mathbf{_2}$}}

\newcommand{\KF}{\variedad{K}_F}
\newcommand{\KR}{\variedad{K}_R}
\newcommand{\clase}[1]{[\![#1]\!]}

\newtheorem{theorem}{Theorem}[section]
\newtheorem{lemma}[theorem]{Lemma}
\newtheorem{prop}[theorem]{Proposition}
\newtheorem{corollary}[theorem]{Corollary}

\theoremstyle{definition}
\newtheorem{question}[theorem]{Question}
\newtheorem{definition}[theorem]{Definition}
\newtheorem{example}[theorem]{Example}
\theoremstyle{remark}

\newcommand{\keywords}[1]{{\renewcommand{\thefootnote}{\relax}\footnotetext{\emph{Keywords:}
      #1.}}}
\newcommand{\MSC}[1]{{\renewcommand{\thefootnote}{\relax}\footnotetext{\emph{MSC 2020:}
      #1.}}}

\usepackage{framed,subcaption}

\hypersetup{
	pdftitle={Left adjoint and factor congruences},
	pdfauthor={Joel Kuperman, Pedro Sánchez Terraf},
	colorlinks,
	breaklinks,
	linkcolor={blue!35!black},
	citecolor={blue!35!black},
	urlcolor={blue!35!black}  
}

\begin{document}
\title{The associative-poset point of view on right regular bands}
\author{Joel Kuperman\thanks{
    Universidad Nacional de Córdoba. 
    Facultad de Matemática, Astronomía,  Física y
    Computación. \\
    \indent \phantom{*} Centro de Investigación y Estudios de Matemática (CIEM-FaMAF).
    Córdoba. Argentina.\\
    Supported by Secyt-UNC project 33620230100751CB and Conicet PIP project 11220210100508CO.}
  \and
  Pedro Sánchez Terraf\footnotemark[1]
}
\maketitle
\begin{abstract}
  We present two results on the relation between the class of right regular
  bands (RRBs) and their underlying
  \emph{associative posets}. The first one is a construction of a left
  adjoint to the forgetful functor that takes an RRB $(P,\cdot)$ to
  the corresponding $(P,\leq)$. The construction of such a left
  adjoint is actually done in general for any class of relational structures $(X,R)$
  obtained from a variety, where $R$ is defined by a finite conjunction of identities.
  The second result generalizes the  “inner” representations of direct product
  decompositions  of semilattices  studied by the second
  author to RRBs having at least one commuting element.
\end{abstract}
\MSC{06F05; 20M07}
\section{Introduction}
\label{sec:introduction}

This paper is a continuation of our joint work with A.~Petrovich
\cite{2024arXiv240407877K}, where we study right regular bands (RRB) $(P,\cdot)$
using as a guiding concept the partial order which arise naturally from the
stipulation
\begin{equation}
  \label{eq:order_of_RRB}
  x \leq y \iff  x\cdot y=x.
\end{equation}
for all $x,y$ in $P$. The poset $(P,\leq)$ thus obtained is said to be
\emph{associative}.

We present two results on the relation between the class of RRBs and that of
associative posets. The first one is motivated by the question on how to obtain
a band from a given poset. To give some context, \emph{choosing} an RRB
structure for a
given associative poset is not trivial, since this necessarily involves the
Axiom of Choice as it was seen in \cite[Thm.~6.2]{2024arXiv240407877K}.
The answer we present here is to construct a left
adjoint to the obvious forgetful functor that takes an RRB $(P,\cdot)$ to
$(P,\leq)$ using (\ref{eq:order_of_RRB}). The construction of such a left
adjoint is actually generalized to any class of relational structures $(X,R)$
obtained from a variety using a finite conjunction of identities to define
$R$. This is presented in Section~\ref{sec:an-adjunction}, where we also discuss some
other concrete cases of this construction (Section~\ref{sec:examples}); notably, there is a corresponding
adjunction between the category of RRBs and that of sets with an equivalence
relation (Example~\ref{exm:equiv-relations}).

The second result (Section~\ref{sec:fact-congr-RRB}) studies direct product
representations of RRBs and generalizes \cite{fact_slat} by the second
author. The main motivation is to obtain “inner” representation of direct
factors and factor congruences as they are generally obtained for some other
algebraic structures as groups and rings (by way of direct summands). The reason
to present this result here stems from the fact that the characterization of
congruences and factors is informed by the underlying partial order. This was
accomplished for RRBs having a non trivial “group-theoretic center;” that is, RRBs
having at least one commuting element.

We conclude in Section~\ref{sec:conclusion} with some open questions.
\section{Preliminaries}
\label{sec:preliminaries}
In this section we fix notation and introduce some basic notions.

We remind the reader that a \emph{right regular band} (RRB) is an idempotent
semigroup satisfying
\begin{equation}
  \label{eq:aba-ba} a · b · a = b · a.
\end{equation}
As mentioned in Section~\ref{sec:introduction}, a poset $(P,\leq)$ is
said to be associative if there is an RRB operation such that
(\ref{eq:order_of_RRB}) holds. We then use the phrase \emph{right posemigroup} when
referring to the expanded structure $(P,\leq, \cdot)$; this stems from the fact
that the product is order-preserving on the right only (see
Lemma~\ref{lem:lower-bounds}.\ref{item:monot-prod}).
Conversely,
given an RRB $(P,\cdot)$, we say that the partial order $x \leq y$
given by (\ref{eq:order_of_RRB}) is the \emph{underlying order} of
this RRB.

We state two simple lemmas whose proofs are left as exercises for the
interested reader, or can be checked at \cite[Sect.~2]{2024arXiv240407877K}.
\begin{lemma}\label{lem:aba-ba}
 In every right posemigroup,
  \begin{enumerate}
  \item \label{item:1} $a · b \leq b$; in particular, if $b$ is
    minimal, $a · b = b$.
  \item \label{item:3} $a \leq b \implies b· a =a$.\qed
  \end{enumerate}
\end{lemma}
\begin{lemma}\label{lem:lower-bounds}
  For every right posemigroup, 
  \begin{enumerate}
  \item \label{item:monot-prod}$a\leq b$ implies $ a · x \leq b · x$.
  \item \label{item:iguales-inf} $c\leq x,y $ implies $c\leq x·y,
    y·x$. Hence if $x·y= y·x$, it must be  the infimum of $\{x,y \}$.\qed
  \end{enumerate}
\end{lemma}

\section{An adjunction}
\label{sec:an-adjunction}
  
The categories of right regular bands and associative posets are canonically
linked by a forgetful functor. An adjunction between them can be constructed
explicitly by making use of the algebraic structure of RRBs and associative
posets. Instead of only doing this, in this section we present a more general result,
by constructing left adjoints for a broad family of forgetful functors that
includes the one mentioned above.

We fix a first-order language $L$ with an $n$-ary relation symbol $R$.
\begin{definition}\label{def:definable-by-conj-ids}
  Let $\K$ be a
  class of $L$-structures. We will say that $R$ is \emph{definable by
  conjunctions of identities} over $\K$ if there exist $L$-terms
  $t_1(x_1,\dots,x_n)$, \dots, $t_m(x_1,\dots,x_n)$,
  $s_1(x_1,\dots,x_n)$, \dots, $s_m(x_1,\dots,x_n)$ such
  that
  \[
    \mathcal{K}\vDash \forall x_1,\dots,x_n,\,
    R(x_1,\dots,x_n)\iff \textstyle\bigwedge_{1\leq i\leq m}
    t_i(x_1,\dots,x_n)=s_i(x_1,\dots,x_n).
  \]
\end{definition}

\begin{definition}
  Let $\K$ be a
  class of $L$-structures. We denote the set of function symbols of $L$ by $L_{\mathrm{alg}}$, and denote the reducts of $\K$ to $L_{\mathrm{alg}}$ and
  to $\{R\}$ by $\K_{\mathrm{alg}}$ and $\K_R$, respectively.
\end{definition}
From now on, $\variedad{K}$ will denote a  class of $L$-structures such that $R$ is definable by
conjunctions of identities over $\variedad{K}$. Throughout this section we will study the interplay between the categories
$\categoria{K}_{\mathrm{alg}}$ and $\categoria{K}_R$, which consist of the classes
$\variedad{K}_{\mathrm{alg}}$ and $\variedad{K}_R$ and their homomorphisms, respectively.
\begin{example}\label{ex:r_def}
	\label{exm:posemigroups-as-defd-by-conj}
	Posemigroups provide an instance of
	Definition~\ref{def:definable-by-conj-ids}. Let $L:=\{\cdot,\leq\}$ and $\variedad{K}$ be the class of right posemigroups. Here $R={\leq}$ is definable by one identity: 
	\[
	\mathcal{K}\vDash \forall x_1,x_2,\ 
	x_1\leq x_2\iff  x_1\cdot x_2 = x_1
	\]
	 In this example, $\categoria{K}_{\mathrm{alg}}$ is the variety of right
         regular bands and $\categoria{K}_R$ is the category of associative posets.
\end{example}

There is a canonical forgetful
functor $U:\categoria{K}_{\mathrm{alg}}\rightarrow\categoria{K}_R$ which assigns to every
$\mathbf{A}\in \K_{\mathrm{alg}}$ the structure
$\mathbf{A}_R:=(A,R^{\mathbf{A}})$ where
\[
  \lb a_1,\dots,a_n\rb\in R^{\mathbf{A}} \iff \mathbf{A}\vDash
  \textstyle\bigwedge_{1\leq i\leq m} t_i(a_1,\dots,a_n)=s_i(a_1,\dots,a_n);
\]
and to every $\categoria{K}_{\mathrm{alg}}$-morphism $f:\A\rightarrow \mathbf{B}$,
the $\categoria{K}_R$-morphism $Uf:\mathbf{A}_R\rightarrow
\mathbf{B}_R$ given by $Uf(a)=f(a)$ (that is, $U$ preserves morphism
as maps).

For the situation described in Example~\ref{ex:r_def}, the functor $U$ maps each RRB to its underlying associative poset.

We will show:
\begin{theorem}
  If $\K_{\mathrm{alg}}$ is a variety, the forgetful functor
  $U:\categoria{K}_{\mathrm{alg}}\rightarrow \categoria{K}_R$ has a left adjoint
  $F$.
\end{theorem}

We will start by defining the functor $F$ on objects.

For each set $X$, let's consider the $\K_{\mathrm{alg}}$-free algebra over $X$, 
denoted by $\bF_{\variedad{K}}(X)$. We construe $\bF_{\variedad{K}}(X)$ as the
quotient of the term algebra $T(X)$, by the least congruence $\delta_X$ such that
the quotient belongs to $\K_{\mathrm{alg}}$. For $t \in T(X)$, we will denote by
$\clase{t}$ the equivalence class of $t$ under that congruence. In particular,
$\{\clase{x}: x \in X \}$ is the set of free generators of $\bF_{\variedad{K}}(X)$.
    
We know that $\bF_{\variedad{K}}(X)$ has the following universal property: For every $\B \in \K_{\mathrm{alg}}$ and any function $\alpha: X \rightarrow B$, there exists a homomorphism $\beta: \bF_{\variedad{K}}(X) \rightarrow \B$ such that $\alpha(x) = \beta(\clase{x})$ for every $x \in X$.
We will use the following particular application: 
\begin{prop}\label{prop-univ}
  Let $X$ and $Y$ be sets and $f:X\rightarrow Y$ a function. Then the map
  given by   $\tilde{f}(\clase{t(\Bar{x})})=\clase{t(f(\Bar{x}))}$ is
  a morphism $\tilde{f}:\bF_{\variedad{K}}(X)\rightarrow
  \bF_{\variedad{K}}(Y)$. \qed
\end{prop}
\begin{definition}[$F$ on objects]\label{def:F-on-objects}
  For $\X \in \K_R$, let 
  $\theta_{\X}$ be the congruence on $\bF_{\variedad{K}}(X)$ generated by
  \[
    \{\lb\clase{t_i(x_1,\dots,x_n)},\clase{s_i(x_1,\dots x_n)}\rb \mid 1\leq
      i\leq m, \lb x_1,\dots,x_n\rb\in R^{\X} \}.
  \]
  The algebra $F\X$ is  the quotient $\bF_{\variedad{K}}(X)/\theta_{\X}$.
\end{definition}
We will also need Gr\"{a}tzer's
version of  Mal'cev's characterization of compact congruences.
\begin{lemma}\label{malsev}
  Let $\A$ be any algebra and let $c,d\in A,$
  $\vec{a},\vec{b}\in
  A^{j}.$ Then $\lb c,d\rb\in \Cg ^{A}(\vec{a},\vec{b})$ if and only if 
  there
  exist $(j+m)$-ary terms 
  $r_{1}(\vec{x},\vec{u}),\dots,r_{k}(\vec{x},\vec{u})$,
  with $k$ odd, and $\vec{z}\in A^{m}$ such that: 
  \begin{align*}
      c& =r_{1}(\vec{a},\vec{z}) \\
      r_{l}(\vec{b},\vec{z})&= r_{l+1}(\vec{b},\vec{z}) && l\text{
        odd,} \\
      r_{l}(\vec{a},\vec{z})&= r_{l+1}(\vec{a},\vec{z}) && l\text{
        even,} \\
      r_{k}(\vec{b},\vec{z})&= d.
  \end{align*}
\end{lemma}
The following result completes the definition of $F$, now on arrows of $\categoria{K}_R$. 
\begin{lemma}[$F$ on morphisms]\label{lem:ext-F}
  Let $\X,\Y\in\K_R$ and $f:\X\rightarrow \Y$ an homomorphism. The function $Ff:F\X\rightarrow F\Y$ given by $Ff(\clase{t(\Bar{x})}/\theta_{\X})=\clase{t(f(\bar{x}))}/\theta_{\Y}$ is well defined and is a morphism.
\end{lemma}
\begin{proof}
  Let $f$ be as in the assumptions, and let $\tilde{f}:\bF_{\variedad{K}}(X)\rightarrow \bF_{\variedad{K}}(Y)$  be the
  morphism provided by  Proposition~\ref{prop-univ} for $f$.
  In particular, we have $\tilde{f}(\clase{x})=\clase{f(x)}$ for every $x\in X$.
  Now, take $c,d\in \bF_{\variedad{K}}(X)$ (where, $c=\clase{t(\bar{x})}$
  and $d=\clase{s(\Bar{x})}$ for some $s,t\in T(X)$) such that $\lb c,d\rb\in\theta_{\X}$. If we can show that $\lb \tilde{f}(c),\tilde{f}(d)\rb\in\theta_{\Y}$, we can then define $Ff(\clase{t(\bar{x})}/\theta_{\X}):=\tilde{f}(\clase{t(\bar{x})})/\theta_{\Y}$, and by how we have chosen $\tilde{f}$, this function will satisfy what we want.
  Since $\lb c,d\rb \in\theta_{\X}$ is witnessed by finitely many tuples from the
  generators of $\theta_{\X}$, namely
  $\lb\clase{t_{i_p}(x_1^p,\dots,x_n^p)},\clase{s_{i_p}(x_1^p,\dots, x_n^p)}\rb$
  for some appropriate $x_l^p\in X$ (with $l=1,\dots, n$, $p=1,\dots, j$), 
  Lemma~\ref{malsev} gives us $(j+m)$-ary terms
  $r_{1}(\vec{x},\vec{u})$,\dots, $r_{k}(\vec{x},\vec{u})$,
  with $k$ odd,
  and $\vec{z}\in
  \bF_{\variedad{K}}(X)^{m}$ such that:
  \begin{align}
    c&=r_{1}(\clase{\bar{W}},\vec{z}),\label{eq:malcev-init}\\ 
    r_{l}(\clase{\bar{V}},\vec{z})&=r_{l+1}(\clase{\bar{V}},\vec{z})
    &&\text{for odd $l$,}\label{eq:malcev-odd} \\
    r_{l}(\clase{\bar{W}},\vec{z})&=r_{l+1}(\clase{\bar{W}},\vec{z})
    &&\text{for even $l$,}\label{eq:malcev-even} \\
    r_{k}(\clase{\bar{V}},\vec{z})&=d\label{eq:malcev-end}
  \end{align}
  where 
  \begin{align*}
    \bar{\mathrm{x}} &\defi x_1^1,\dots,x_n^1,\dots,x_1^j,\dots,x_n^j\\
    \bar{W} &= \bar{W}(\bar{\mathrm{x}}) \defi t_{i_1}(x_1^1,\dots,x_n^1),\dots,t_{i_j}(x_1^j,\dots,x_n^j)\\
    \bar{V} &= \bar{V}(\bar{\mathrm{x}}) \defi s_{i_1}(x_1^1,\dots,x_n^1),\dots,s_{i_j}(x_1^j,\dots,x_n^j),
  \end{align*}
  with the following abuses of notation: We take equivalence classes term-wise,
  \[
    \clase{\bar{W}}=\clase{t_{i_1}(x_1^1,\dots,x_n^1)},\dots,\clase{t_{i_j}(x_1^j,\dots,x_n^j)},
  \]
  and apply functions in the same way:
  \begin{align*}
    h(\bar{W}) &=  h\bigl(t_{i_1}(x_1^1,\dots,x_n^1)\bigr),\dots,h\bigl(t_{i_j}(x_1^j,\dots,x_n^j)\bigr)\\
    \bar{W}\bigl(h(\bar{\mathrm{x}})\bigr) &= t_{i_1}\bigl(h(x_1^1),\dots,h(x_n^1)\bigr),\dots,t_{i_j}\bigl(h(x_1^j),\dots,h(x_n^j)\bigr)
  \end{align*}

  Using the fact that $\lb f(x_1^i), \dots, f(x_n^i)\rb \in
  R^{\Y}$ for every $i$, we will 
  show that the terms obtained  witness the fact that  $\lb\tilde{f}(c), \tilde{f}(d)\rb \in
  \theta_{\Y}$ by transitivity.
  \begin{align*}
    \tilde{f}(c)&=\tilde{f}\bigl(r_{1}(\clase{\bar{W}(\bar{\mathrm{x}})},\vec{z})\bigr)&&\text{by (\ref{eq:malcev-init})}\\ 
    &=r_1\bigl(\tilde{f}(\clase{\bar{W}(\bar{\mathrm{x}})}),\tilde{f}(\bar{z})\bigr)&&\text{$\tilde{f}$ is a morphism}\\
    &=r_1\bigl(\clase{\bar{W}(f(\bar{\mathrm{x}}))},\tilde{f}(\bar{z})\bigr).&&\text{by definition of $\tilde{f}$}
  \end{align*}
  Observe now that $\lb
  r_1(\clase{\bar{W}(f(\bar{\mathrm{x}}))},\tilde{f}(\bar{z})),
  r_1(\clase{\bar{V}(f(\bar{\mathrm{x}}))},\tilde{f}(\bar{z}))\rb \in
  \theta_\Y$, where the second term appears for $l=1$ below. In general, for odd $l$,   
  \begin{align*}
    r_l\bigl(\clase{\bar{V}(f(\bar{\mathrm{x}}))},\tilde{f}(\bar{z})\bigr)
    &=r_l\bigl(\bar{V}(\clase{f(\bar{\mathrm{x}})}),\tilde{f}(\bar{z})\bigr)&&\text{$\delta_Y$ is a congruence}\\
    &=r_l\bigl(\bar{V}(\tilde{f}(\clase{\bar{\mathrm{x}}})),\tilde{f}(\bar{z})\bigr) &&\text{by definition of $\tilde{f}$}\\
    &=r_l\bigl(\tilde{f}(\bar{V}(\clase{\bar{\mathrm{x}}})),\tilde{f}(\bar{z})\bigr) &&\text{$\tilde{f}$ is a morphism}\\
    &=r_l\bigl(\tilde{f}(\clase{\bar{V}(\bar{\mathrm{x}})}),\tilde{f}(\bar{z})\bigr) &&\text{$\delta_X$ is a congruence}\\
    &=\tilde{f}\bigl(r_l(\clase{\bar{V}(\bar{\mathrm{x}})},\bar{z})\bigr) &&\text{$\tilde{f}$ is a morphism}\\
    &=\tilde{f}\bigl(r_{l+1}(\clase{\bar{V}(\bar{\mathrm{x}})},\bar{z})\bigr) && \text{by (\ref{eq:malcev-odd})}\\
    &=r_{l+1}\bigl(\tilde{f}(\clase{\bar{V}(\bar{\mathrm{x}})}),\tilde{f}(\bar{z})\bigr) &&\text{$\tilde{f}$ is a morphism}\\
    &=r_{l+1}\bigl(\tilde{f}(\bar{V}(\clase{\bar{\mathrm{x}}})),\tilde{f}(\bar{z})\bigr) &&\text{$\delta_X$ is a congruence}\\
    &=r_{l+1}\bigl(\bar{V}(\tilde{f}(\clase{\bar{\mathrm{x}}})),\tilde{f}(\bar{z})\bigr) &&\text{$\tilde{f}$ is a morphism}\\
    &=r_{l+1}\bigl(\clase{\bar{V}(f(\bar{\mathrm{x}}))},\tilde{f}(\bar{z})\bigr)&&\text{by definition of $\tilde{f}$}.
  \end{align*}
  In general, $\lb
  r_{l+1}(\clase{\bar{W}(f(\bar{\mathrm{x}}))},\tilde{f}(\bar{z})),
  r_{l+1}(\clase{\bar{V}(f(\bar{\mathrm{x}}))},\tilde{f}(\bar{z}))\rb \in
  \theta_\Y$, for $l\leq k-1$. Similarly, for even $l$,
  \begin{align*}
    r_l\bigl(\clase{\bar{W}(f(\bar{\mathrm{x}}))},\tilde{f}(\bar{z})\bigr)
    &=r_l\bigl(\tilde{f}(\clase{\bar{W}(\bar{\mathrm{x}})}),\tilde{f}(\bar{z})\bigr) &&\text{by definition of $\tilde{f}$}\\
    &=\tilde{f}\bigl(r_l(\clase{\bar{W}(\bar{\mathrm{x}})}),\bar{z}\bigr) &&\text{$\tilde{f}$ is a morphism}\\
    &=\tilde{f}\bigl(r_{l+1}(\clase{\bar{W}(\bar{\mathrm{x}})},\bar{z})\bigr) && \text{by (\ref{eq:malcev-even})}\\
    &=r_{l+1}\bigl(\tilde{f}(\clase{\bar{W}(\bar{\mathrm{x}})}),\tilde{f}(\bar{z})\bigr) &&\text{$\tilde{f}$ is a morphism}\\
    &=r_{l+1}\bigl(\clase{\bar{W}(f(\bar{\mathrm{x}}))},\tilde{f}(\bar{z})\bigr)&&\text{by definition of $\tilde{f}$}.
  \end{align*}
  Finally,
  \begin{align*}
    r_k\bigl(\clase{\bar{V}(f(\bar{\mathrm{x}}))},\tilde{f}(\bar{z})\bigr)
    &=r_k\bigl(\bar{V}(\clase{f(\bar{\mathrm{x}})}),\tilde{f}(\bar{z})\bigr) &&\text{$\delta_Y$ is a congruence}\\
    &=r_k\bigl(\bar{V}(\tilde{f}(\clase{\bar{\mathrm{x}}})),\tilde{f}(\bar{z})\bigr) &&\text{by definition of $\tilde{f}$}\\
    &=r_k\bigl(\tilde{f}(\bar{V}(\clase{\bar{\mathrm{x}}})),\tilde{f}(\bar{z})\bigr) &&\text{$\tilde{f}$ is a morphism}\\
    &=r_k\bigl(\tilde{f}(\clase{\bar{V}(\bar{\mathrm{x}})}),\tilde{f}(\bar{z})\bigr) &&\text{$\delta_X$ is a congruence}\\
    &=\tilde{f}\bigl(r_k(\clase{\bar{V}(\bar{\mathrm{x}})},\bar{z})\bigr) &&\text{$\tilde{f}$ is a morphism}\\
    &=\tilde{f}(d)&&\text{by (\ref{eq:malcev-end}).}
  \end{align*}
  Therefore the function $Ff$
  is well defined. Now take $q$ an $n$-ary operation symbol in
  $L_{\mathrm{alg}}$. We have that
\begin{align*} 
	Ff\bigl(q^{F\X}(\clase{u_1}/\theta_{\X},\dots, \clase{u_n}/\theta_{\X})\bigr)
	&= \tilde{f}\bigl(q^{\bF_{\variedad{K}}(X)}(\clase{u_1},\dots,\clase{u_n})/\theta_{\X}\bigr)\\
	&=q^{\bF_{\variedad{K}}(Y)}\bigl(\tilde{f}(\clase{u_1}),\dots, \tilde{f}(\clase{u_n})\bigr)/\theta_{\Y}\\
	&= q^{F\Y}\bigl(\tilde{f}(\clase{u_1})/\theta_{\Y},\dots,\tilde{f}(\clase{u_n})/\theta_{\Y}\bigr)\\
	&=q^{F\Y}\bigl(Ff(\clase{u_1}/\theta_{\X}),\dots, Ff(\clase{u_n}/\theta_{\X})\bigr)
 \end{align*}
 so $Ff$ is indeed a morphism.
\end{proof}

For a structure $\X\in\categoria{K}_R$, we define an inclusion $\eta_{\X}: \X\to
UF\X$ by $\eta_{\X}(a)\defi \clase{a}/\theta_{\X}$.  The previous lemma allows
us to prove the following:
\begin{lemma}
  \label{lem:adjunto}
  For every $\X$ and $\Y$ in $\categoria{K}_R$, and $\categoria{K}_R$-morphism $h:\Y\rightarrow \X$:
  \begin{equation}\label{eq:adjunto}
    UFh\circ\eta_\Y=\eta_\X\circ h.
  \end{equation}
\end{lemma}
\begin{proof}
  Let $y \in Y$. We have:
  $(\eta_{\X} \circ h)(y) = \eta_{\X}(h(y)) =\clase{h(y)}/\theta_{\X}.$
  Notice that, by the definition of $Fh$, this element is equal to $Fh(\clase{y}/\theta_{\Y}) = \bigl(Fh \circ \eta_{\Y}\bigr)(y)$. And since $U$ preserves morphisms as functions, we have that this is equal to $\bigl(UFh \circ \eta_{\Y}\bigr)(y)$.
\end{proof}
For a structure $\X \in \categoria{K}_R$, we define %
a congruence $D_{\X}$ on
$T(X)$ as follows: $\lb t(\bar{x}), s(\bar{x})\rb \in D_{\X}$ if and only if
$t^{\A}(\bar{x}) = s^{\A}(\bar{x})$ for every $\A \in \K_{\mathrm{alg}}$ such
that $U \A = \X$. It is clear that $\delta_X\sbq D_{\X}$.
\begin{lemma}\label{lem:proy-lib}
  Given a structure $\X\in \categoria{K}_R$, and an algebra $\A\in \KF$ such
  that $U \A= \X$ 
  there exists a morphism $g:F\X\rightarrow \A$ such that $g(\eta_{\X}(x))=x$ for every $x\in X$.
\end{lemma}
\begin{proof}
  Note there is a morphism $f:T(X)/D_{\X}\rightarrow \A$ for which
  $f(x/D_{\X})=x$, given by $f(t(\bar{x})/D_{\X})\defi t^\A(\bar{x})$; this
  is well defined by definition of $D_{\X}$.

  We have that
  $\theta_{\X}\subseteq D_{\X}/\delta_X$: For every $\bar{x}\in X^n$ such
  that $\bar{x}\in R^{\X}$ we have that
  $\lb t_i(\bar{x}),s_i(\bar{x})\rb\in D_{\X}$ for $1\leq i\leq
  m$ and hence
  \[
    \{ \lb \clase{t_i(\bar{x})},
    \clase{s_i(\bar{x})}\rb:1\leq i\leq m, \bar{x}\in
    R^{\X} \}\subseteq D_{\X}/\delta_X
  \]
  by definition of $D_{\X}/\delta_X$. Since $\theta_{\X}$ is generated by those
  pairs, we obtain the inclusion.

  We now know that there is a projection morphism
  \[
    \pi:(T(X)/\delta_X)/\theta_{\X}\rightarrow
    ((T(X)/\delta_X)/\theta_{\X})/((D_{\X}/\delta_X)/\theta_{\X})
  \]
  which satisfies
  \[
    \pi(\clase{x}/\theta_{\X})=(\clase{x}/\theta_{\X})/\bigl((D_{\X}/\delta_X)/\theta_{\X}\bigr).
  \]
  By the Second Isomorphism Theorem, there are isomorphisms
  \[
    h_1:\bigl((T(X)/\delta_X)/\theta_{\X}\bigr)/\bigl((D_{\X}/\delta_X)/\theta_{\X}\bigr)\rightarrow
    (T(X)/\delta_X)/(D_{\X}/\delta_X)
  \]
  and
  \[      
    h_2:(T(X)/\delta_X)/(D_{\X}/\delta_X)\rightarrow T(X)/D_{\X}
  \]
  which satisfy
  \[
    h_1\bigl((\clase{x}/\theta_{\X})/\bigl((D_{\X}/\delta_X)/\theta_{\X}\bigr)\bigr)=\clase{x}/(D_{\X}/\delta_X)
  \]
  and
  \[       
    h_2\bigl(\clase{x}/(D_{\X}/\delta_X)\bigr)=x/D_{\X}.
  \]
  We can now take $g\defi f\circ h_2\circ h_1\circ \pi:F\X\rightarrow A$
  which is a morphism satisfying $g(\eta_{\X}(x))=x$.
\end{proof}

\begin{lemma}
  $F:\categoria{K}_{R}\rightarrow \categoria{K}_{\mathrm{alg}}$ is a functor.
\end{lemma}
\begin{proof}
  Given $\X,\Y,\mathbf{Z}\in\categoria{K}_{R}$, and $f:\X\rightarrow \Y$,
  $g:\Y\rightarrow \mathbf{Z}$ $\categoria{K}_{R}$-morphisms, we show that
  $Fg\circ Ff=F(g\circ f)$. Let $x\in F\X$ and assume 
  $x=\clase{t(\bar{x})}/\theta_{\X}$ for some $t$ and $\bar{x}$. We have
  \[
    (Fg\circ
    Ff)(x)=Fg(\clase{t(f(\bar{x}))}/\theta_{\Y})=t(gf(\bar{x}))/R_Z)/\theta_Z)=F(g\circ
    f)(x).
  \]

  Now let $\X\in\KR$; we show that $F1_{\X}=1_{F\X}$. 
  Take $x=\clase{t(\bar{x})}/\theta_{\X}\in F\X$. We have that
  \[
    F1_{\X}(x)=\clase{t(1_{\X}(\bar{x}))}/\theta_{\X}=\clase{t(\bar{x})}/\theta_{\X}=x. \qedhere
  \]
\end{proof}
Consider now the map $\phi_{\X,\A}:
\categoria{K}_{\mathrm{alg}}(F\X,\A)\to\categoria{K}_{R}(\X,U\A)$  between
homsets given by $\phi_{\X,\A}(f)=Uf\circ \eta_{\X}$.

\begin{theorem}\label{th:adjunction-general}
  The triple $\langle F, U, \phi \rangle$ is an adjunction from $\categoria{K}_{R}$ to $\categoria{K}_{\mathrm{alg}}$.
\end{theorem}
\begin{proof}
  We first prove that $\phi_{\X,\A}$ is injective. Let $g: F\X \rightarrow \A$
  and $h: F\X \rightarrow \A$ be two $\KF$-morphisms such that $\phi_{\X,\A} (h)
  = \phi_{\X,\A} (g)$, i.e.,
  \[
    Uh(\eta_{\X}(x)) = (Uh \circ \eta_{\X}) (x) = (Ug \circ \eta_{\X}) (x) =
    Ug(\eta_{\X}(x))
  \]
  for all $x \in \X$. From this we deduce that $g$ and $h$ coincide on all elements
  of $F\X$ of the form $\clase{x}/\theta_{\X} = \eta_{\X}(x)$: We have that
  $h(\eta_{\X}(x))$ equals $Uh(\eta_{\X}(x))$ since $U$ preserves morphisms as
  functions. By hypothesis, $Uh(\eta_{\X}(x)) = Ug(\eta_{\X}(x))$, the latter
  being equal to $g(\eta_{\X}(x))$ for the same reason. Therefore, $h(\eta_{\X}(x)) =
  g(\eta_{\X}(x))$ for all $x \in \X$. Now, let $y \in F\X$ be of the form $x =
  \clase{t(x_1,\dots,x_k)}/\theta_{\X}$. Since $g$ and $h$ are morphisms, we
  have:
  \begin{align*}
    g(y)&=g(\clase{t(x_1,\dots,x_k)}/\theta_{\X})\\
    &=t^\A\bigl(g(\clase{x_1}/\theta_{\X}),\dots,g(\clase{x_k}/\theta_{\X})\bigr)\\
    &=t^\A\bigl(h(\clase{x_1}/\theta_{\X}),\dots,h(\clase{x_k}/\theta_{\X})\bigr)\\
    &=h(\clase{t(x_1,\dots,x_k)}/\theta_{\X})\\
    &=h(y),
  \end{align*}
  which means that $g = h$.

  Next we show that $\phi_{\X,\A}$ is surjective.  Let $f:\X\rightarrow U\A$ be
  a $\categoria{K}_{R}$-morphism. Take $Ff:F\X\rightarrow FU\A$ the
  $\categoria{K}_{\mathrm{alg}}$-morphism given by Lemma~\ref{lem:ext-F}, and compose it
  with the morphism $g:FU\A\rightarrow \A$ given by
  Lemma~\ref{lem:proy-lib}. Now we have $\phi_{\X,\A}(g\circ Ff)(x)=U(g\circ
  Ff)\circ \eta_{\X} (x)$. Since $U$ preserves morphisms as functions, this
  element is equal to $g(Ff (\eta_{\X}(x)))$ which, by the definition of $Ff$,
  is equal to $g(\eta_{U\A}(f(x)))$, and this last term is equal to $f(x)$ by
  construction of $g$. Therefore, $f=\phi_{\X,\A}(g\circ Ff)$.

  Now let's see that $\phi$ satisfies the naturality conditions. That is, we
  want to show that for every pair of objects $\X\in\categoria{K}_{R}$ and
  $\A\in\categoria{K}_{\mathrm{alg}}$, every
  $\categoria{K}_{\mathrm{alg}}$-morphism $f:F\X\rightarrow \A$, every
  $\categoria{K}_{R}$-morphism $h:\Y\rightarrow \X$, and every
  $\categoria{K}_{\mathrm{alg}}$-morphism $k:\A\rightarrow \B$, it holds that
  $\phi_{\X,\B}(k\circ f)=Uk\circ \phi_{\X,\A}(f)$ and $\phi_{\Y,\A}(f\circ
  Fh)=\phi_{\X,\A} f\circ h$.

  On the one hand, we have $\phi_{\X,\B}(k\circ f)=U(k\circ f)\circ\eta_{\X}=Uk\circ
  Uf\circ\eta_{\X}=Uk\circ \phi_{\X,\A} (f)$, as desired. 

  On the other hand, we have $\phi_{\Y,\A}(f\circ Fh)=U(f\circ
  Fh)\circ\eta_{\Y}=Uf\circ UFh\circ\eta_{\Y}=Uf\circ\eta_{\X}\circ h=\phi_{\X,\A} (f)\circ
  h$, since $UFh\circ\eta_{\Y}=\eta_{\X}\circ h$ by Equation~\ref{eq:adjunto} from Lemma~\ref{lem:adjunto}.

  Therefore, the naturality conditions are satisfied. Thus, $\phi$ is a natural bijection between $\categoria{K}_{\mathrm{alg}}(F\X,\A)$ and $\categoria{K}_{R}(\X,U\A)$, and hence the triple $\langle F,U,\phi\rangle$ forms an adjunction.
\end{proof}

\begin{corollary}\label{cor:adjunction-RRB-AP}
Let $\categoria{RRB}$ and $\categoria{AP}$ be the categories of RRBs with semigroup homomorphisms and
associative posets with non decreasing maps,
respectively. Then the forgetful functor $U: \categoria{RRB}\to \categoria{AP}$ admits a left
adjoint.
\end{corollary}
\begin{proof}
  By Theorem~\ref{th:adjunction-general} and Example~\ref{exm:posemigroups-as-defd-by-conj}.
\end{proof}
\subsection{Examples}
\label{sec:examples}

We now present some examples. In these examples, we omit the subscript in $\eta_\mathbf{X}$, as it will be clear from context.
\begin{example}\label{example:1}
	In Figure~\ref{fig:adjunto} we can see an example of the partial order
        underlying $F\mathbf{P}$ for a simple poset $\mathbf{P}=(P,\leq)$. Each
        product of the two maximal elements in $UF\mathbf{P}$ is one of the two
        elements in the middle level (immediately below the second factor).
	\begin{figure}[h] 
		\begin{center}
			\begin{tabular}{c@{\hspace{4em}}c}
				\begin{tikzpicture}
					[>=latex, thick,
					nodo/.style={thick,minimum size=0cm,inner sep=0cm}]

					\node (x) at (0,1) [label=left:$x$] [nodo] {$\bullet$};
					\node (0) at (1,0) [label=right:$0$] [nodo] {$\bullet$};
					\node (y) at (2,1)  [label=right:$y$] [nodo] {$\bullet$};
					\draw [-] (x) edge  (0);
					\draw [-] (y) edge  (0);
					
				\end{tikzpicture}
				&
				\begin{tikzpicture}
					[>=latex, thick,
					nodo/.style={thick,minimum size=0cm,inner sep=0cm}]

					\node (x) at (0,1) [label=left:$\eta(x)$]    [nodo] {$\bullet$};
					\node (b) at (2,0)  [nodo] {$\bullet$};
					\node (0) at (1,-1) [label=right:$\eta(0)$]  [nodo] {$\bullet$};
					\node (a) at (0,0)  [nodo] {$\bullet$};
					\node (y) at (2,1)  [label=right:$\eta(y)$]  [nodo] {$\bullet$};
					\draw [-] (x) edge  (a);
					\draw [-] (y) edge  (b);
					\draw [-] (a) edge  (0);
					\draw [-] (b) edge  (0);
					
				\end{tikzpicture} \\
				$\mathbf{P}$ & $UF\mathbf{P}$
			\end{tabular}
		\end{center}
		\caption{}\label{fig:adjunto}
	\end{figure}
\end{example}

\begin{example}\label{exm:equiv-relations}
In this example we show that an appropriate expansion $\mathcal{K}$ of the class
of right regular bands has $\mathcal{K}_R$ equal to the class of sets equipped
with some equivalence relation.

Let $L := \{\cdot, \theta\}$ and let $\mathcal{K}$ be the class satisfying all the right regular band axioms together with the additional condition
\[
\forall x,y,\quad (x \mathrel{\theta} y) \Longleftrightarrow (x \cdot y = y) \wedge (y \cdot x = x).
\]
We have that $\mathcal{K}_{\text{alg}}$ is, once again, the class of
right regular bands. It can be proved that $\theta$ is always an equivalence
relation. 

Moreover, let $\delta$ be an equivalence relation over a set $S$. Index the equivalence classes of $\delta$ by a linearly ordered set $I$ and denote them by $\{x_i \mid i \in I\}$. Define an RRB operation $\cdot_i$ on each $x_i$ by
\[
a \cdot_i b := b \quad \text{for all } a,b \in x_i.
\]
Now the ordered sum of the $x_i$ admits an RRB operation \cite[Lemma~4.1]{2024arXiv240407877K}. It is clear that $a \cdot b = b \wedge b \cdot a = a$ holds if and only if $a,b \in x_i$ for some $i \in I$.

To illustrate this, we consider the three-element RRB $\mathbf{A}$ given by the
multiplication table in Figure~\ref{fig:2} (observe that the partial order
underlying $\mathbf{A}$ is isomorphic to the poset $\mathbf{P}$ from the
previous example). We have that $U\mathbf{A}$ is the set $\{a,b,0\}$ together
with the equivalence relation given by the partition $\{\{a,b\},\{0\}\}$. We
also provide the Hasse diagram describing the partial order underlying
$FU\mathbf{A}$. Keep in mind that this partial order holds no relation to
the original partial order.
      
\begin{figure}
	\centering
	\resizebox{\textwidth}{!}{
	\begin{tabular}[t]{c@{\hspace{2em}}c@{\hspace{2em}}c}
		\begin{tabular}{c|ccc}
			$·$ & $0$ & $a$ & $b$\\\hline
			$0$ & $0$ & $0$ & $0$\\
			$a$ & $0$ & $a$ & $b$\\
			$b$ & $0$ & $a$ & $b$
		\end{tabular}
		&
		\begin{tabular}{c|ccccccc}
			$·$ &  $a$ & $b$ & $c$ & $w$ & $x$ & $y$ & $z$ \\\hline
			$a$ & $a$ & $b$ & $y$ & $w$  & $x$ & $y$ & $z$ \\
			$b$ & $a$ & $b$ & $z$ & $w$ & $x$ & $y$ & $z$ \\
			$c$ & $w$ & $x$ & $c$ & $w$ & $x$ & $y$ & $z$ \\
			$w$ & $w$ & $x$ & $y$ & $w$ & $x$ & $y$ & $z$ \\
			$x$ & $w$ & $x$ & $z$ & $w$ & $x$ & $y$ & $z$ \\
			$y$ & $w$ & $x$ & $y$ & $w$ & $x$ & $y$ & $z$ \\
			$z$ & $w$ & $x$ & $z$ & $w$ & $x$ & $y$ & $z$
		\end{tabular} & \begin{tikzpicture}[>=latex, thick,
		nodo/.style={thick,minimum size=0cm,inner sep=0cm}]

		\node (a0) at (0,0) [nodo] {$\bullet$};
		\node (a1) at (0,1) [label=above:$\eta(a)$] [nodo] {$\bullet$};
		\draw (a0) -- (a1);

		\node (b0) at (1,0) [nodo] {$\bullet$};
		\node (b1) at (1,1) [label=above:$\eta(b)$] [nodo] {$\bullet$};
		\draw (b0) -- (b1);

		\node (0) at (2.5,1) [label=above:$\eta(0)$] [nodo] {$\bullet$};
		\node (c) at (2,0)  [nodo] {$\bullet$};
		\node (w) at (3,0)  [nodo] {$\bullet$};
		\draw (0) -- (c);
		\draw (0) -- (w);
		\end{tikzpicture}
		 \\
		$\mathbf{A}$ & $FU\mathbf{A}$ & The partial order underlying $FU\mathbf{A}$
	\end{tabular}
}
\caption{}\label{fig:2}
\end{figure}

\end{example}

\begin{example}
	Let $L:= \{\cdot, \leq\}$ and $\mathcal{K}$ the class of commutative right posemigroups. $\mathcal{K}_{\text{alg}}$ is the class of meet-semilattices (as algebras) and $\mathcal{K}_R$ is the class of meet-semilattices (as posets).

	Let $\mathbf{B}$ be the poset $\mathbf{P}$ in Example~\ref{example:1}, regarded as a meet-semilattice.
	Then $U\mathbf{B}$ is that same poset and $FU\mathbf{B}$ is the four element meet-semilattice given by the Hasse diagram in Figure~\ref{fig:FUB}.
	  \begin{figure}[h]
		\begin{center}
			\begin{tikzpicture}
				[>=latex, thick,
				nodo/.style={thick,minimum size=0cm,inner sep=0cm}]
				
				\node (x) at (-.5,2) [label=left:$a$] [nodo] {$\bullet$};
				\node (y) at (.5,2) [label=left:$b$] [nodo] {$\bullet$};
				\node (c) at (0,1) [label=left:$c$][nodo] {$\bullet$};   
				\node (0) at (0,0) [label=left:$0$][nodo] {$\bullet$}; 
				\draw [-] (x) edge  (c);
				\draw [-] (y) edge  (c);
				\draw [-] (0) edge  (c);
			\end{tikzpicture}
			\caption{$FU\mathbf{B}$}\label{fig:FUB}
		\end{center}
	\end{figure}
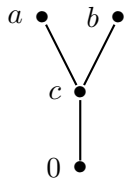
Observe that $\mathbf{B}$ can also be thought of as an RRB, since semilattices
are a subvariety of the variety of RRB's. In that case, we would obtain the
(different) $F\mathbf{P}$ from Example~\ref{example:1}. This shows that the
definition of the functor $F$ depends not only on the structure to which it is
applied but also on the equational class.
\end{example}
\begin{example}\label{exm:distr-complement}
	Let $L := \{\wedge, \vee, 0, 1, C\}$ and $K$ the class of bounded distributive lattices, along with the binary relation defined by ``being a pair of complemented elements''. That is 
	\[
	\mathcal{K} \models \forall x,y, C(x,y)\iff (x\wedge y = 0) \;\&\; (x\vee y=1)
	\] 
Note that this relation is symmetric and irreflexive in every non-trivial lattice. Therefore, it can be represented as a graph. In this example, we consider the three-element chain as a lattice and show, in Figure~\ref{fig:4}, the graph corresponding to $U\mathbf{C}$ alongside the Hasse diagram of $FU\mathbf{C}$.
	
		\begin{figure}[h] 
		\begin{center}
			\begin{tabular}{c@{\hspace{4em}}c@{\hspace{4em}}c}
				\begin{tikzpicture}
					[>=latex, thick,
					nodo/.style={thick,minimum size=0cm,inner sep=0cm}]

					\node (1) at (0,2) [label=left:$1$] [nodo] {$\bullet$};
					\node (x) at (0,1) [label=right:$x$] [nodo] {$\bullet$};
					\node (0) at (0,0)  [label=left:$0$] [nodo] {$\bullet$};
					\draw [-] (x) edge  (0);
					\draw [-] (x) edge  (1);
					
				\end{tikzpicture}
				&
				\begin{tikzpicture}
					[>=latex, thick,
					nodo/.style={thick,minimum size=0cm,inner sep=0cm}]

					\node (1) at (0,1) [label=left:$1$] [nodo] {$\bullet$};
				\node (x) at (0.5,0) [label=right:$x$] [nodo] {$\bullet$};
				\node (0) at (1,0.8)  [label=right:$0$] [nodo] {$\bullet$};
				\draw [-] (1) edge  (0);

				\end{tikzpicture} 
				&
					\begin{tikzpicture}[>=latex, thick,
						nodo/.style={thick,minimum size=0cm,inner sep=0cm}]

						\node (00) at (0,0) [nodo] {$\bullet$};

						\node (10) at (-1,1) [nodo] {$\bullet$};
						\node (01) at (1,1)  [nodo] {$\bullet$};

						\node (20) at (-2,2) [label=left:$\eta(0)$] [nodo] {$\bullet$};
						\node (11) at (0,2)  [label=right:$\eta(x)$][nodo] {$\bullet$};
						\node (02) at (2,2) [label=right:$\eta(1)$] [nodo] {$\bullet$};

						\node (21) at (-1,3)  [nodo] {$\bullet$};
						\node (12) at (1,3) [nodo] {$\bullet$};

						\node (22) at (0,4) [nodo] {$\bullet$};

						\draw (00) -- (10);
						\draw (00) -- (01);
						
						\draw (10) -- (20);
						\draw (10) -- (11);
						
						\draw (01) -- (11);
						\draw (01) -- (02);
						
						\draw (20) -- (21);
						\draw (11) -- (21);
						\draw (11) -- (12);
						\draw (02) -- (12);
						
						\draw (21) -- (22);
						\draw (12) -- (22);
						
					\end{tikzpicture}
					
				\\
				$\mathbf{C}$ & $U\mathbf{C}$ & $FU\mathbf{C}$
			\end{tabular}
		\end{center}
		\caption{}\label{fig:4}
	\end{figure}
Observe that any graph obtained in this way necessarily includes a two-vertex
component (namely, the one formed by the vertices corresponding to $0$ and
$1$).
\end{example}

\subsection{SAFT applicability}

Although one might expect a result such as Corollary~\ref{cor:adjunction-RRB-AP}
to follow from general categorical principles (such as the Special Adjoint
Functor Theorem \cite[Sect.~V.8]{lane1998categories}), this is not the case as we now show. Specifically, the
adjunction we construct cannot be derived from that theorem, as the category of
RRBs lacks a cogenerating set, a situation common to residually large idempotent
varieties. For completeness, we include a proof of this last fact.
\begin{lemma}
  Let $\mathcal{V}$ a idempotent variety. If $\mathcal{V}$ has a cogenerating
  set, then it is residually small.
\end{lemma}
\begin{proof}
Suppose there exists a cogenerating set $Q$. Let $\mathbf{B}\in \mathcal{V}$. Let $\Hom(\mathbf{B},\mathbf{Q})$ be the set of homomorphisms from $\mathbf{B}$ to $\mathbf{Q}$. Define $\mathbf{P}$ as \[ \mathop{\mathbf{P}:= \prod_{\mathbf{Q}\in Q}\prod_{g\in \Hom(\mathbf{B},\mathbf{Q})}\mathbf{Q}}\] and let $f:\mathbf{B}\rightarrow \mathbf{P}$ be defined coordinate-wise as $f(x)_{\mathbf{Q},g} = g(x)$. We claim that $f$ is an embedding from $\mathbf{B}$ into $\mathbf{P}$. Indeed, let $x,y \in \mathbf{B}$ with $x\neq y$. Let $\{*\}$ be the trivial algebra and define $h,h':\{*\}\rightarrow \mathbf{B}$ as the homomorphisms with image $\{x\}$ and $\{y\}$ respectively. Since $Q$ is a cogenerating set, there exist $\mathbf{Q}\in Q$ and $g:\mathbf{B}\rightarrow \mathbf{Q}$ such that $g\circ h\neq g\circ h'$ hence $g(x)\neq g(y)$. Therefore, we have that $f(x)_{\mathbf{Q},g}\neq f(y)_{\mathbf{Q},g}$ and hence $f(x)\neq f(y)$.

By Birkhoff's Theorem there exists a suitable set and a subdirect embedding $\alpha_{\mathbf{Q}}: \mathbf{Q}\rightarrow \prod_{i\in I_{\mathbf{Q}}}\mathbf{A}_i$. Now, we can define a new subdirect embedding \[ \mathop{\alpha: \mathbf{P}\longrightarrow \prod_{\mathbf{Q}\in Q}\prod_{ g\in \Hom(\mathbf{B},\mathbf{Q})} \prod_{ i\in I_{\mathbf{Q}}}\mathbf{A}_i}\] given by $\alpha(x)_{\mathbf{Q},g,i}=\alpha_{\mathbf{Q}}(x_{\mathbf{Q},g})_i$. Define $\mathcal{S}_{\mathbf{Q},g,i}$ to be the image of $\pi_{\mathbf{Q},g,i}\circ \alpha \circ f$ where\[ \mathop{\pi_{\mathbf{Q},g,i}: \prod_{\mathbf{Q}\in Q}\prod_{ g\in \Hom(\mathbf{B},\mathbf{Q})} \prod_{j\in I_{\mathbf{Q}}}\mathbf{A}_j}\longrightarrow \mathbf{A}_i\] is the canonical projection for all $\mathbf{Q}\in Q$, $g\in \Hom(\mathbf{B},\mathbf{Q})$, $i\in I_{\mathbf{Q}}$. Note that each $\mathcal{S}_{\mathbf{Q},g,i}$ is a substructure of $\mathbf{A}_i$ and that the function \[ \mathop{\beta: b \longrightarrow \prod_{ \mathbf{Q}\in Q}\prod_{ g\in \Hom(\mathbf{B},\mathbf{Q})}\prod_{ i\in I_{\mathbf{Q}}}}\mathcal{S}_{\mathbf{Q},g,i}\] given by $\beta(x)_{\mathbf{Q},g,i}=(\pi_{\mathbf{Q},g,i}\circ \alpha \circ f)(x)$ is a subdirect embedding.

By the previous discussion, we have that each $\mathbf{B}\in \mathcal{V}$ is isomorphic to a subdirect product of algebras in $\mathcal{S}(\{\mathbf{A}_i\mid \mathbf{Q}\in Q, i\in I_{\mathbf{Q}}\})$ where this last family is the set of substructures of some $\mathbf{A}_i$ with $\mathbf{Q}\in Q$ and $i\in I_{\mathbf{Q}}$. Therefore, every subdirectly irreducible algebra must be isomorphic to some member in  $\mathcal{S}(\{\mathbf{A}_i\mid \mathbf{Q}\in Q, i\in I_{\mathbf{Q}}\})$. This shows that $\mathcal{V}$ must be residually small.
\end{proof}

\begin{lemma}\label{lem:RRB-large}
  The variety of RRBs is residually large.
\end{lemma}
The proof presented here follows an example given in \cite{Ger2}. To proceed, we will rely on an auxiliary result, whose proof is provided in the cited source. We begin by defining two fundamental concepts.
\begin{definition}
\begin{itemize}
    \item A \emph{zero} of a semigroup $\mathbf{S}$ is an element $0\in \mathbf{S}$ which satisfies $0\cdot s = s\cdot 0 = 0$
    for all $s\in \mathbf{S}$
    \item Given a semigroup $\mathbf{S}= (S, \cdot)$, its dual $\mathbf{S}^*$ is the semigroup with multiplication given by $x\tilde{\cdot} y = y\cdot x$ for all $x,y\in S$.
\end{itemize}
    \end{definition}
\begin{theorem}\label{teo:subdirect}
    An idempotent semigroup $S$ (without zero) is subdirectly irreducible if and only if $\mathbf{S}$ or its dual $\mathbf{S}^*$ is isomorphic to a semigroup $\mathbf{T}$ which satisfies the following two conditions:
 \begin{itemize}

	\item $C(X) \subseteq \mathbf{T} \subseteq X^X$, where $X^X$ is the semigroup of all mappings of $X$ into itself, $C(X)$ is the set of all constant mappings of $X$, and $\mathbf{T}$ is a subsemigroup of $X^X.$
	\item There exist $k, k' \in C(X)$ such that $\theta(k,k')\subseteq \theta(c, d)$ for all $c, d \in C(X)$, $c\neq d$, where $\theta(x,y)$ denotes the smallest congruence containing $(x,y)$.\qed
\end{itemize}
\end{theorem}
\begin{proof}[Proof of Lemma~\ref{lem:RRB-large}]
    We will show that subdirectly irreducible right regular bands of arbitrarily large cardinality exist. Fix an arbitrary set $X$ with two distinguished elements $0$ and $1$. Let $A$ consist of the following functions from $X$ into itself: 
\begin{itemize}
\item The mappings $\{a_i \mid i\in X\}$ are the constant mappings defined by $a_i(j)=i$. 
\item The mappings $\{b_i \mid i\in X, i\notin\{0,1\}\}$ are defined by $b_i(i)=i$, $b_i(j)=0$ if $i\neq j$. 
\item The mappings $\{c_i \mid i\in X, i\notin\{0,1\} \}$ are defined by $c_i(i)=i$, $c_i(j)=1$ if $i\neq j$. 
\end{itemize}
Define multiplication on $A$ by $a\cdot b=b\circ a$ where $\circ$ denotes composition of functions. Since the $a_i, b_i$, and $c_i$ are idempotent, we can show that it is an idempotent semigroup by proving that it is closed with respect to the composition of functions. If one of the factors in a given composition is a constant, the result is a constant. If both are non-constants, they occur among the following: $c_i\circ b_i=c_i$, $b_i\circ c_i = b_i$, and if $i\neq j$, $c_j\circ b_i = c_j\circ c_i = a_1$, $b_j\circ b_i = b_j\circ c_i = a_0$. We will now consider the dual $\mathbf{A}^*$, where multiplication is given by composition of functions. We will show that this semigroup satisfies the conditions of  Theorem~\ref{teo:subdirect}, which tells us that $\mathbf{A}$ is subdirectly irreducible (note that since $a_i\cdot a_j= a_j$ for all $i,j$, $\mathbf{A}$ has no zero). By the mentioned theorem, it is enough to show that $\theta (a_0,a_1)\subseteq \theta (a_i,a_j)$ for all $i\neq j$. If $i\neq j$, $\theta(a_i,a_j)$ contains the pairs $(b_i\circ a_i, b_i\circ a_j)=(a_i,a_0)$ and $(c_i\circ a_i, c_i\circ a_j)=(a_i, a_0)$, and therefore it must contain $(a_0,a_1)$, thus $\mathbf{A}$ is a subdirectly irreducible idempotent semigroup. We now show that $\mathbf{A}$ satisfies $x\cdot y\cdot x = y\cdot x)$. If $x, y$ or $x\cdot y$ is a constant, or if $x = y$, the result is trivial. The only other case is $\{x, y\} = \{b_i, c_j\}$ for some $i,j$ and in that case it is easy to check that $x\cdot y\cdot x = y\cdot x$.
\end{proof}
We conclude that there's no cogenerating set for the variety of RRBs, so
the Special Adjoint Functor Theorem cannot be applied in this case.
\section{Factor congruences of right regular bands}
\label{sec:fact-congr-RRB}
In the following, we provide an application of the order-theoretic perspective
on bands to the study of direct product decompositions of RRBs with a central
element. This generalizes the results from \cite{fact_slat}.

One key idea
from that paper (which studies join-semilattices) was that existing binary infima in finite direct
products of join-semilattices must factorize and satisfy some
distributive and absorption properties with respect to the join.
Since the $(\leq,\cdot)$-posemigroups associated to commutative bands correspond to
meet-semilattices (Lemma~\ref{lem:lower-bounds}(\ref{item:iguales-inf})), all concepts here will be dual those in
\cite{fact_slat}; in particular, we will be speaking of partial binary suprema. We
have the analogous
\begin{lemma} \label{lem:join_prod}
  For any pair of RRBs $C,D$ and elements $c,e \in C$ and $d,f
  \in D$, the element $ \lb c, d\rb \supr \lb e, f\rb $ exists in $C
  \times D$ if and only if $c \supr e$ and $d \supr f$ exist.
\end{lemma}

The main
difference between this section and
\cite{fact_slat} is that, unlike $\vee$ nor $\wedge$, the RRB operation
$\cdot$ is not commutative in general. We do assume for the rest of this section that at least one element
commutes with every other:
\begin{definition}
  Given an RRB $(A,\cdot)$, an element $c\in A$ is said to be
  \emph{central} if for all $x\in A$, we have $x\cdot c=c\cdot x$.
\end{definition}
We will highlight modifications to the proofs in the sequel. 

\subsection{Decomposition into direct product of RRBs}
For the remainder of this chapter, we will write formulas in the
language $ \{ ·, \supr \} $; occurrences of $\leq$ can be reduced using
(\ref{eq:order_of_RRB}).
The formula “$x \supr y = z$”
will be interpreted as “$z$ is the supremum of $\{ x, y \}$”:
\[ x,y\leq z\mathrel{\&} (\forall u :  x,y\leq u \;\limp\; z\leq
  u), \] unless otherwise specified; in general, every equation $t_1 =
t_2$ involving $\supr$ should be interpreted as “if one of the terms
exists, so does the other and they are equal.” It is easy to check
that the laws of associativity hold for the partial operation $\supr$
in every poset.

From now
on, let $A$ be an RRB and $c \in A$ be a fixed central
element. Let $\phi(c,x_1, x_2, x) $ be the conjunction of the
following formulas:
\begin{quote}
  \begin{itemize}  
  \item [\comm] $ x \cdot x_1 = x_1 \cdot x$ and $ x \cdot x_2 = x_2
    \cdot x$.
  \item [\dist] $ x = (x \cdot x_1) \vee (x \cdot x_2). $
  \item [\puno] $ x_1 = (x \cdot x_1) \vee (c \cdot x_1). $
  \item [\pdos] $ x_2 = (x \cdot x_2) \vee (c \cdot x_2). $
  \item [\produ] $x_1 \cdot x_2 = x \cdot c.$
  \end{itemize} 
\end{quote}
Note that \comm{} is the only formula that does not appear in the
formula $\phi(c,x_1, x_2, x) $ defined in \cite{fact_slat}. Note also
that, due to the assumption that $c$ is central and 
\comm, Lemma~\ref{lem:lower-bounds}(\ref{item:iguales-inf}) guarantees
that all products appearing in $\phi$ except for the last one are
actually infima. We will write “$x=\llangle x_1,x_2\rrangle_c$” to
denote that $\phi(c,x_1,x_2,x)$ holds.
\begin{definition}\label{defn:direct_sum}
  Suppose that $I_1, I_2$ are subsemigroups of $A$. We will say that
  $A$ is the \emph{$c$-direct product} of $I_1$ and $I_2$, and write
  $A=I_1 \times_c I_2$, if and only if the following conditions hold:
  \begin{quote}
    \begin{itemize}  
    \item [\perm] The elements of $I_1$ commute with those of $I_2$.
    \item [\Modi] For all $x,y\in A$, $x_1\in I_1$, and $x_2 \in I_2$,
      if $x\cdot c \leq x_1\cdot x_2$, then
      \begin{align*}
        \bigl((x \cdot x_1) \supr (x \cdot x_2)\bigr) \cdot y = (x
        \cdot x_1\cdot y) \supr (x \cdot x_2\cdot y),\\
        y \cdot
        \bigl((x \cdot x_1) \supr (x \cdot x_2)\bigr) = (y \cdot x
        \cdot x_1) \supr (y \cdot x \cdot x_2).
      \end{align*}
    \item [\Modii] For all $x,y\in A$, $x_1\in I_1$, and$x_2 \in I_2$,
      if $x \geq x_1\cdot x_2$, then
      \begin{align*}
        \bigl((x \cdot x_i) \supr (c \cdot x_i)\bigr) \cdot y = (x
        \cdot x_i \cdot y) \supr (c \cdot x_i \cdot y),\\
        y \cdot \bigl((x \cdot x_i) \supr (c
        \cdot x_i)\bigr) = (y \cdot x \cdot x_i ) \supr (y \cdot c
        \cdot x_i ).
      \end{align*}
      for $i=1,2$.
    \item [\Abs] For all $x_1, y_1 \in I_1$, and $z_2 \in I_2$, we
      have: $x_1 \supr (y_1 \cdot z_2) = x_1 \supr (y_1 \cdot c) $
      (and exchanging the roles of $I_1$ and $I_2$).
    \item [\exi] $ \forall x_1 \in I_1, x_2 \in I_2 \; \exists x \in
      A: x=\llangle x_1, x_2\rrangle_c. $
    \item [\onto] $ \forall x \in A \ \exists x_1 \in I_1, x_2 \in I_2
      : \; x=\llangle x_1, x_2\rrangle_c.$
    \end{itemize}
  \end{quote}
\end{definition}
Note that these are “two-sided” versions of the conditions in
\cite{fact_slat}, plus the commutativities \perm{}.

For the sake of brevity, we will omit the reference to
$c$, writing $x=\llangle x_1,x_2\rrangle$ instead of $x=\llangle
x_1,x_2\rrangle_c$. Note that \exi{} implies:
\begin{quote}
  \begin{itemize}
  \item[\ori] $ \forall x_1 \in I_1, x_2 \in I_2 \; (x_1 \cdot x_2 \leq c)$.
  \end{itemize}
\end{quote}

\subsection{The representation lemma}
\begin{lemma}\label{l:dir_sum} Suppose $A = I_1 \times_c I_2$. Then
  $x=\llangle x_1,x_2\rrangle$ defines an isomorphism $\lb x_1,x_2\rb
  \stackrel{\phi}{\mapsto} x$ between $I_1 \times I_2$ and $A$.
\end{lemma}
\begin{proof}
  The proof that the function $\lb x_1, x_2\rb \stackrel
  {\phi}{\mapsto} x $ is well-defined is very similar to that in
  \cite{fact_slat}, but it involves essential uses of \perm{}
  and \comm{}. It is left as an interesting exercise for the reader.

  The function defined by $\phi$ is surjective by \onto; to see it is injective,
  let $x=\llangle x_1,x_2\rrangle$ and $x=\llangle
  y_1,y_2\rrangle$. We have:
  \begin{align*}  
    x_1 &= x_1 \supr (x_1 · x_2) \\
    &= x_1 \supr (x · c) && \text{by
      \produ} \\
    &= x_1 \supr (y_1 · y_2) && \text{by \produ{}
      again} \\
    &= x_1 \supr (y_1 · c) && \text{by \Abs}
  \end{align*}  
  and then $x_1 \geq y_1 · c$. Similarly, $y_1 \geq x_1 · c$ and in
  conclusion, by Lemma ~\ref{lem:lower-bounds}(\ref{item:1}),
  \begin{equation}\label{eq: 1} 
    x_1 · c = y_1 · c.
  \end{equation}  
  On the other hand,
  \begin{align*} 
    x · y_1 &=y_1 · x &&\text{by \comm{}} \\
    &= y_1 · ((x · x_1) \supr
    (x · x_2 ))&& \text {by \dist} \\
    &= (y_1 · x · x_1) \supr (y_1 · x
    · x_2 ) && \text {by \Modi} \\
    &= (x · y_1 · x_1) \supr (x · y_1 ·
    x_2) && \text{by \comm{}} \\
    &= (x · y_1 · x_1) \supr (x · y_1 · x_2
    · c) && \text{by \ori} \\
    &= (x · y_1 · x_1) \supr (x · c) &&
    \text{$c$ central and $x · c \leq x_1, y_1, x_2, y_2$} \\
    &= x · y_1 ·
    x_1, && \text{by the same argument.}
  \end{align*}  
  and also
  \begin{align*} 
    x · x_1 &=((x · y_1) \supr (x · y_2 )) · x_1&& \text {by \dist}
    \\
    &= (x · y_1 · x_1) \supr (x · y_2 · x_1) && \text {by \Modi}
    \\
    &= (x · y_1 · x_1) \supr (x · x_1 · y_2) && \text{by \perm{}}
    \\
    &= (x · y_1 · x_1) \supr (x · x_1 · y_2 · c) && \text{by \ori}
    \\
    &= (x · y_1 · x_1) \supr (x · c) && \text{$c$ central and $x · c
      \leq x_1, y_1, x_2, y_2$} \\
    &= x · y_1 · x_1. && \text{by the
      same argument.}
  \end{align*}  
  Thus, we obtain $x · x_1= x · y_1$. (Here, we must use the fact
  that $x · y_1=y_1 · x$ because otherwise, it is not possible to prove
  that $x · y_1=x · x_1$). Combining this with~(\ref{eq: 1}) and using \puno{},
  we have
  \[x_1 = (x · x_1 ) \supr (c · x_1) = (x · y_1 ) \supr (c · y_1) = y_1.\]
  By the same reasoning, $x_2=y_2$. The preceding part requires a more
  extensive development than the proof in \cite{fact_slat}.

  Next we
  prove that $\phi$ preserves $·$. Suppose $ x=\llangle
  x_1,x_2\rrangle$ and $z=\llangle z_1, z_2\rrangle$; since each $I_1, I_2$
  is a subsemigroup, we know that $x_j · z_j \in I_j$ for
  $j=1,2$. We want to show $ x · z=\llangle x_1 · z_1, x_2 ·
  z_2\rrangle$. The property \produ{} is immediate (and therefore, we can apply
  \Modi{} and \Modii{} to $x· z$).  Now we prove \dist:
  \begin{align*}  
    x · z &= \bigl((x · x_1) \supr (x · x_2)\bigr) · z && \text {by
      \dist{} for $x$} \\
    &= (x · x_1 · z ) \supr (x · x_2 · z ).  &&
    \text {by \Modi}
  \end{align*}
  This last term is equal to
  \begin{equation}\label{eq:7}
    (x · x_1 · z · z_1) \supr (x · x_1 · z · z_2) \supr (x · x_2 · z ·
    z_1) \supr (x · x_2 · z · z_2),
  \end{equation}
  by \dist{} for $z$ and \Modi.  Note that, for any $y$,
  \begin{align*}
    (y · z · z_1)&=y · z · (z · z_1\supr c · z_1) &&\text{by \puno{} for
      $z_1$}\\
    &=(y · z · z_1)\supr (y · z · c · z_1) &&\text{by
      \Modi{} for $z$}\\
    &=(y · z · z_1)\supr (y · z_1 · z_2 · z_1)
    &&\text{by \produ{} for $z$}\\
    &=(y · z · z_1)\supr (y · z_2 ·
    z_1) &&\text{by Eq.~\ref{eq:aba-ba}}\\
    &=(y · z · z_1)\supr
    (z_2 · y · z_1) &&\text{by \perm{}}\\
    &=(y · z · z_1)\supr (z_1 ·
    z_2 · y · z_1) &&\text{by Eq.~\ref{eq:aba-ba}
      again}\\
    &=(y · z · z_1)\supr (z · c · y · z_1) &&\text{by
      \produ{} for $z$}\\
    &=(z · y · z · z_1)\supr (z · y · c · z_1)
    &&\text{by Eq.~\ref{eq:aba-ba}; $c$ is central}\\
    &=z ·
    y · (z · z_1\supr c · z_1) &&\text{by \Modi{} for $z$}\\
    &=z · y ·
    z_1.
  \end{align*}
  This development is necessary in order to permute $x_1$ and $x_2$ with $z$
  in the first and fourth term of~(\ref{eq:7}) respectively, as we cannot
  guarantee that $\cdot$ commutes over these elements. This equality will be
  used multiple times in this proof, always for the same reason. Similarly,
  $z · y · z_2=y · z · z_2$.

  We can now rewrite the term~(\ref{eq:7}) as follows:
  \begin{align*}
    (x · z · x_1 · z_1) \supr (x · x_1 · z · z_2)\supr (x · x_2 · z · z_1)
    \supr (x · z · x_2 · z_2).
  \end{align*}

  Note that
  \begin{align*}  
    x · x_1 · z · z_2 &= x · x_1 · z_2 · z && \text{by \comm{}} \\
    &= x ·
    x_1 · z_2 · c · z &&\text{by \ori}\\
    &= x · c · x_1 · z_2 · z_1 · z_2 &&\text{by Eq.~\ref{eq:aba-ba} and $c$ is central}\\
    &= x_1 ·
    x_2 · x_1 · z_2 · z_1 · z_2 &&\text{by \produ{}}\\
    &= x_2 ·
    x_1 · z_1 · z_2\\
    &= x_1 · x_2 · z_1 · z_2. &&\text{by \perm{}}
  \end{align*}  
  Similarly, $x · x_2 · z · z_1=x_1 · x_2 · z_1 · z_2$

  Let's see that $(x_1 · x_2 · z_1 · z_2)\leq(x · x_1 · z · z_1)$:
  \begin{align*}
    (x_1 · x_2 · z_1 · z_2) · (x · z · x_1 · z_1) &= x_2 · z_2 · x ·
    x_1 · z · z_1\\ &=x_1 · x_2 · z_2 · x · z · z_1&&\text{by \comm{},
      \perm{}}\\
    &=x · c · z_2 · x · z · z_1 &&\text{by \produ{} for $x$}\\
    &=z_2 · x · z · c · z_1 \\
    &=z_2 · x · c · z
    · c · z_1 \\
    &=z_2 · x_1 · x_2 · z_1 · z_2 · z_1 &&\text{by \produ{} for $z$}\\
    &= x_1 · x_2 · z_2 · z_1 \\
    &= x_1 · x_2 · z_1 · z_2.
  \end{align*}
  Therefore, we can eliminate the two middle terms in Equation~(\ref{eq:7}) and obtain \dist{} for $x· z$:
  \[
    (x · z)  = \bigl ((x · z)  · (x_1 · z_1) \bigr) \supr \bigl
    ((x · z) · (x_2 · z_2) \bigr).
  \]
  Here, the method for eliminating the two terms also differs from
  what was done in \cite{fact_slat}. We can obtain \puno{} and \pdos{}
  in a similar way. We prove \puno:
  \begin{align*}  
    x_1 · z_1 &= ((x · x_1) \supr (c · x_1))· z_1 && \text{by \puno{}
      for $x$} \\
    & = (x · x_1 · z_1) \supr (c · x_1 · z_1). && \text{by
      \Modii}
  \end{align*}
  By \puno{} for $z$ followed by \Modii{} on each term of the supremum, the last term is equal to
  \begin{equation*}
    ( x · x_1 · z · z_1) \supr (x · x_1 · c · z_1) \supr (c · x_1 · z ·
    z_1) \supr (x_1 · c · z_1 ).
  \end{equation*}
  Due to the observed fact (that $z · x_1 · z_1=x_1 · z · z_1$) and
  the fact that $c$ is central, we can rewrite this term as
  \begin{equation}
    ( x · z · x_1 · z_1) \supr (x · c · x_1 · z_1) \supr (x_1 · z · c ·
    z_1) \supr (c · x_1 · z_1 ). \label{eq: 6}
  \end{equation}
  Note that $(x · c · x_1 · z_1)=x · (c · x_1 · z_1)\leq(c · x_1 ·
  z_1)$, so we can eliminate the second term of the supremum. To
  eliminate the third term, we can rewrite it as:
  \begin{align*}
    (x_1 · z · c · z_1)&=(x_1 · z_1 · z_2 · z_1) &&\text{By \produ{}
      for $z$}\\
    &=(x_1 · z_2 · z_1)\\
    &=(x_1 · z · c) &&\text{by
      \produ{} for $z$}\\
    &=(c · x_1 · z). &&\text{$c$ is central}
  \end{align*}
  Let's see that $(c · x_1 · z)\leq (c · x_1 · z_1)$:
  \begin{align*}
    (c · x_1 · z) · (c · x_1 · z_1)
    &= z · c · x_1 · z_1\\
    &=z_1 · z_2 · x_1 · z_1 &&\text{by \produ{} for $z$}\\
    &=x_1 · z_1  · z_2 &&\text{by \perm{}}\\
    &=x_1 · z · c &&\text{by \produ{} for $z$}\\
    &=c · x_1 · z. &&\text{$c$ is central}
  \end{align*}
  Then Equation~(\ref{eq:6}) becomes:
  \begin{align*}
    ( x · z · x_1 · z_1) \supr (c · x_1 · z_1 )
  \end{align*}
  and this is equal to
  \[
    (  ( x · z) · (x_1 · z_1)) \supr (c · (x_1 · z_1 )). \qedhere
  \]
\end{proof}

\subsection{The factorization theorem}
We begin by
introducing some notation.
Recall that  $\ker f$ is 
always a congruence when $f:A\rightarrow B$ is a homomorphism.
We say that two congruences $\theta, \delta\in \Con(A)$ are
\emph{complementary factor congruences} if $\theta\cap\delta=\mathrm{Id}_A$ and
their relational compositions $\theta\circ\delta$ and $\delta\circ\theta$ equal $A\times A$.
Complementary factor congruences are in 1-1 correspondence to direct
product decompositions.

As $\llangle\cdot,\cdot \rrangle$ defines an isomorphism, there exist
canonical projections $\pi_j : A \rightarrow I_j$ with $j=1,2$ such
that:
\begin{equation}\label{eq:iso}
  \forall x\in A, x_1\in I_1, x_2\in I_2 \ : \ x=\llangle
  x_1,x_2\rrangle \iff \pi_1(x) =x_1 \text{ and } \pi_2(x) = x_2.
\end{equation}
Let's define, for any congruence $\theta$ on $A$, the set $I_\theta
\defi \{ a\in A : a\thr c \}$.
\begin{theorem}\label{th:bijection}
  Let $A$ be an RRB and $c\in A$ be a central element. The
  mappings
  \[\begin{array}{ccc}
    \lb \theta, \delta\rb & \stackrel{\mathsf{I}}{\longmapsto} & \lb
    I_\theta, I_\delta\rb \\ \lb \ker \pi_2, \ker \pi_1\rb &
    \stackrel{\mathsf{K}}{\longmapsfrom} & \lb I_1, I_2\rb
    \end{array}\]
  are mutually inverse maps defined between pairs of complementary factor
  congruences of $A$ and the set of pairs of subsemigroups $I_1,
  I_2$ of $A$ such that $A=I_1 \times_c I_2$.
\end{theorem}
\begin{proof}
  The only part of this proof that differs in a nontrivial way from
  \cite{fact_slat} is the verification that \Modi{} and \Modii{} hold.

  The
  mapping $a \mapsto \lb a/\theta, a/\delta\rb $ is an isomorphism
  between $A$ and $ A/\theta \times A/\delta$. Under this isomorphism,
  $I_\theta$ corresponds to $\{\lb c', a''\rb : a'' \in A/\delta \}$ and
  $I_\delta$ corresponds to $\{\lb a', c''\rb : a' \in A/\theta \}$,
  where $c' = c/\theta$ and $c'' =c/\delta$. From now on, we will
  identify $I_\theta$ and $I_\delta$ with their respective isomorphic
  images and verify the axioms for $I_\theta \times_c I_\delta = A$ in $
  A/\theta \times A/\delta$. Note that since $c$ is central, $c'$ and
  $c''$ are also central.

  To verify \Modi{}, suppose $x= \lb x', x''\rb $, $y= \lb y', y''\rb $,
  $x_1= \lb c', x''_1\rb \in I_\theta$, and $x_2 = \lb x'_2,c''\rb \in
  I_\delta$. Notice that $x· c \leq x_1· x_2$ implies
  \begin{align}\label{eq:5}
    x'· c' &\leq c'· x_2' & x''· c'' &\leq x_1''· c''.
  \end{align}
  From the first inequality in~(\ref{eq:5}), we obtain

  \begin{align*}
    ( x'· c') · (x' · x_2')&= x' · x_2' · c' &&\text{$c'$ is
      central}\\ &=x' · c' · x_2 · c'\\ &= (x' · c') · (x_2' ·
    c')\\ &=x' · c'.
  \end{align*}
  In other words, $x'· c' \leq x' · x_2'$. \\ Similarly, from the second
  inequality in~(\ref{eq:5}), we obtain $x''· c''\leq x'' · x_1''$.

  That is:
  \begin{align}
    x'· c' &\leq x' · x_2'  &  x''· c''&\leq x'' · x_1'', \label{eq:6}
  \end{align}
  and therefore, we have $x'· c'· y' \leq x' · x_2'· y'$ and
  $x''· c'' · y'' \leq  x''· x_1''· y''$. Applying Lemma~\ref{lem:join_prod}, we obtain:
  \begin{align*}
    (x  · x_1 \supr x · x_2) · y 
    &=  \lb (x' ·  c'\supr x'· x_2')· y', (x''· x_1''\supr x''·
    c'') ·  y''\rb  \\
    &= \lb x' ·  x_2'· y', x''· x_1'' · y''\rb  \\
    &= \lb x' · c' · y'  \supr x'· x_2'· y', x'' · x_1''· y'' \supr
    x''  · c''·  y''\rb \\
    & = x  · x_1· y \supr x · x_2 · y.
  \end{align*}

  Now, to distribute to the left, the equations~(\ref{eq:5}) also imply
  \begin{align*}
    y' · x'· c' &=y' · x'· c' · x_2' ·c  \\
    &=y' · x' · x_2' ·c  \\
    &=c\cdot y' · x' · x_2'.
  \end{align*}
  This gives us $y' · x'· c' \leq y' · x' · x_2'$ and, analogously, $y'' · x''· c' \leq y'' · x'' · x_1''$.
  Then, together with~(\ref{eq:6}), we have
  \begin{align*}
    y · (x  · x_1 \supr x · x_2) 
    &=  \lb y'· (x' ·  c'\supr x'· x_2'), y''· (x''· x_1''\supr x''·
    c'') \rb  \\
    &= \lb y'· x' ·  x_2', y''·x''· x_1''\rb  \\
    &= \lb y'· x' · c'   \supr y'·x'· x_2', y''·x'' · x_1''\supr
    y''·x''  · c''\rb \\
    & = y· x  · x_1\supr y·x · x_2 .
  \end{align*}

  We leave \Modii{} to the reader.
\end{proof}

As in \cite{fact_slat}, the characterization of direct representations assume a simpler
forms when $c$ is an endpoint of the poset. For example,
\begin{theorem}\label{th:caract_0} Let $A$ be an RRB with identity $1$. Then $A=I_1 \times_1 I_2$  if and
  only if $I_1,  I_2 \leq A$ satisfy:   
  \begin{quote}
    \begin{enumerate}
    \item [\perm] The elements of $I_1$ commute with those of $I_2$.
    \item[\Abs] For all $x_1, y_1 \in I_1$ and $z_2 \in I_2$, we have:
      $x_1 \supr (y_1 \cdot z_2) = x_1 \supr y_1  $ (and interchanging $I_1$ and $I_2$).  
    \item[\onto]  $ I_1 \cdot I_2 = A.$  
    \end{enumerate}  
  \end{quote}
  Moreover, $I_1$ and $I_2$ are filters of $A$.\qed
\end{theorem}

\section{Conclusion}
\label{sec:conclusion}
An order-theoretical point of view while studying bands has proven to be very
informative, as the previous results show. We have also seen some fruits of a
more general, categorical approach to the study of definability of RRBs and
other classes of algebras.  This approach has also
led us to ask some questions which we now present.

In Example~\ref{exm:distr-complement}, we focused on the variety of distributive lattices
because it is locally finite, while the
free (general) lattice on three generators is infinite. As a result, computing $FU\mathbf{C}$
becomes significantly more complicated when considering general
lattices. Observe that the connected components of graphs
obtained by applying $U$ contain at most two elements. This
motivates omitting the distributivity hypothesis:
\begin{question}
  Let $L := \{\wedge, \vee, 0, 1, C\}$ and $K$ the class of bounded (non
    necessarily distributive)
    lattices, where $C$ is the binary relation defined by ``being a pair of
    complemented elements''. Which graphs are in the class $\K_R$?
\end{question}
For a sample we have (in all cases, a two-element connected component is to be
added to the indicated graph):
\begin{itemize}
\item A $3$-cycle, by considering $M_3$.
\item The complete graph $K_n$, by considering $M_n$.
\item In general, the complete bipartite graph $K_{n,m}$ is obtained
  from $\{0\} + (L_n\sqcup L_m) + \{1\}$, where $+$ denotes ordered sum,
      $\sqcup$ denotes disjoint union, and $L_k$ denotes the $k$-element linear order. Note that this allows us to obtain the $4$-cycle.
\end{itemize}
It is not possible to obtain the $5$-cycle.
We do not know if there is an example having it as \emph{one of its components}.

A final last remark is that in our seminal example involving RRBs (Corollary~\ref{cor:adjunction-RRB-AP}) we are
considering plain poset homomorphisms as the morphisms of the category
$\categoria{AP}$. We are very curious if there exists a more appropriate notion
of morphism that preserves more of the structure of associative posets.

\paragraph*{Acknowledgment}
We thank Prof.~Miguel Campercholi (Universidad Nacional de Córdoba) for suggesting that we study the existence of
a left adjoint for the forgetful functor from the category of RRBs to that of
associative posets. We also want to thank Nancy Moyano (Centro de Investigación
y Estudios de Matemática, CIEM-FaMAF) for support during this project.

\providecommand{\noopsort}[1]{}
\begin{small}\end{small}

\end{document}